\documentclass[12pt]{article}
\usepackage{a4wide}
\usepackage{amsthm}
\usepackage{amsfonts}
\usepackage{amssymb}
\usepackage{amsmath}
\usepackage{cite}
\usepackage{epsfig}
\usepackage{stmaryrd}
\newtheorem{theorem}{Theorem}
\newtheorem{lemma}[theorem]{Lemma}
\newtheorem{proposition}[theorem]{Proposition}

\newcommand\GG{{\cal G}}
\newcommand\PP{{\cal P}}
\newcommand\NN{{\mathbb N}}

\newcommand\dd{\mbox{d}}
\DeclareMathOperator{\supp}{supp}
\DeclareMathOperator{\sgn}{sgn}

\begin{document}
\title{The dimension of the feasible region of pattern densities\thanks{An extended abstract announcing the results presented in this paper has been published in the Proceedings of Eurocomb'23.}}

\author{Frederik Garbe\thanks{Universit\"at Heidelberg, Im Neuenheimer Feld 205, 69120 Heidelberg, Germany. E-mail: {\tt garbe@informatik.uni-heidelberg.de}. Previous affiliation: Faculty of Informatics, Masaryk University, Botanick\'a 68A, 602 00 Brno, Czech Republic. Supported by the MUNI Award in Science and Humanities (MUNI/I/1677/2018) of the Grant Agency of Masaryk University.}\and
Daniel Kr{\'a}l'\thanks{Faculty of Informatics, Masaryk University, Botanick\'a 68A, 602 00 Brno, Czech Republic. E-mail: {\tt dkral@fi.muni.cz}. Supported by the MUNI Award in Science and Humanities (MUNI/I/1677/2018) of the Grant Agency of Masaryk University.}
\and
Alexandru Malekshahian\thanks{Department of Mathematics, King's College London. E-mail: {\tt alexandru.malekshahian@kcl.ac.uk}.}
\and
Raul Penaguiao\thanks{Max Planck Institute for the Sciences, Inselstra\ss e 22, 04103 Leipzig, Germany. E-mail: {\tt raul.penaguiao@mis.mpg.de}.}
}

\date{}

\maketitle

\begin{abstract}
A classical result of Erd\H os, Lov\'asz and Spencer from the late 1970s asserts that
the dimension of the feasible region of densities of graphs with at most $k$ vertices in large graphs
is equal to the number of non-trivial connected graphs with at most $k$ vertices.
Indecomposable permutations play the role of connected graphs in the realm of permutations, and
Glebov et al. showed that pattern densities of indecomposable permutations are independent,
i.e., the dimension of the feasible region of densities of permutation patterns of size at most $k$
is at least the number of non-trivial indecomposable permutations of size at most $k$.
However, this lower bound is not tight already for $k=3$.
We prove that the dimension of the feasible region of densities of permutation patterns of size at most $k$
is equal to the number of non-trivial Lyndon permutations of size at most $k$.
The proof exploits an interplay between algebra and combinatorics inherent to the study of Lyndon words.
\end{abstract}

\section{Introduction}
\label{sec:intro}

The interplay between densities of substructures and
particularly determining the extremal points of the region of feasible densities
is a theme underlying a large body of problems in extremal combinatorics.
A fundamental question related to the region of feasible densities, which traces back to Whitney's work~\cite{Whi32} on the number of subgraphs with a given number of vertices and edges, is determining its number of degrees of freedom, i.e., its dimension.
For example, while not obvious at first sight,
the densities of all four $3$-vertex graphs are determined by any two of them.
In the late 1970s,
Erd\H os, Lov\'asz and Spencer~\cite{ErdLS79}
determined the dimension of the region of feasible subgraph/homomorphic densities of graphs.
They showed that homomorphic densities of non-trivial connected graphs are independent and
the density of any other graph is a function of densities of non-trivial connected graphs.
The aim of this short paper is to determine the dimension of the region of feasible pattern densities of permutations,
where we describe a behavior that unexpectedly differs from the graph case in a substantial way.
In particular,
in addition to being ``connected'' (which is captured as being indecomposable in the context of permutations),
the linear order inherent to permutations plays an essential role in the number of degrees of freedom,
which is manifested by a connection to algebraic and combinatorial properties of Lyndon words that we exploit.

We now state the above mentioned results on graphs and our new results on permutations formally
using the language of the theory of combinatorial limits which we use throughout the paper.
In the theory of combinatorial limits,
large graphs are represented by an analytic object called a graphon and
large permutations by an analytic object called a permuton (we refer
the reader to Subsections~\ref{subsec:limg} and \ref{subsec:limp} for definitions as needed).
Let $\GG_k$ be the set of all graphs with at most $k$ vertices and
$\GG^C_k$ be the set of all connected graphs with at least two and at most $k$ vertices.
Erd\H os, Lov\'asz and Spencer~\cite{ErdLS79} showed the following:
\begin{theorem}[{Erd\H os, Lov\'asz and Spencer~\cite{ErdLS79}}]
\label{thm:graph}
For every $k\ge 2$,
there exist $x_0\in [0,1]^{\GG^C_k}$ and $\varepsilon>0$ such that
for every $x\in B_{\varepsilon}(x_0)\subseteq [0,1]^{\GG^C_k}$,
there exists a graphon $W$ such that \[t(G,W)_{G\in\GG^C_k}=x.\]
Moreover,
there exists a polynomial function $f:[0,1]^{\GG^C_k}\to [0,1]^{\GG_k}$ such that
the following holds for every graphon $W$:
\[f(t(G,W)_{G\in\GG^C_k})=t(G,W)_{G\in\GG_k}.\]
\end{theorem}
\noindent In other words,
the densities of non-trivial connected graphs are independent, and
moreover the homomorphic density of any other graph
is a polynomial function of homomorphic densities of connected graphs.
Therefore,
the dimension of the feasible region of homomorphic densities of graphs with at most $k$ vertices in large graphs
is actually equal to the number of non-trivial connected graphs with at most $k$ vertices.

Indecomposable permutations,
i.e., permutations that cannot be expressed as the direct sum of two permutations,
play the role of connected graphs in the realm of permutations, and
it is plausible to assume that the dimension would be equal to the number of indecomposable permutations.
Let $\PP_k$ be the set of all permutations of size at most $k$, and
let $\PP^I_k$ be the set of all non-trivial indecomposable permutations of size at most $k$.
Glebov et al.~\cite{GleHKKKL17} showed that
for every $k\in\NN$,
there exist $x_0\in [0,1]^{\PP^I_k}$ and $\varepsilon>0$ such that
for every $x\in B_{\varepsilon}(x_0)\subseteq [0,1]^{\PP^I_k}$
there exists a permuton $\Pi$ such that $d(\sigma,\Pi)_{\sigma\in\PP^I_k}=x$,
i.e., the densities of indecomposable permutations
are independent in the analogy to the graph case covered by Theorem~\ref{thm:graph}.
However, the bound is not tight already for $3$-point patterns: the dimension is five
although there are only four non-trivial indecomposable permutations of size at most three.

Borga and the last author~\cite{BorP20,BorP23}
studied the dimension of the feasible region of densities of patterns of size at most $k$ and
observed, utilizing a result of Vargas~\cite{Var14},
a link between the dimension and Lyndon permutations.
We say that a permutation is Lyndon if the word formed by its indecomposable blocks
is Lyndon (a formal definition is given in Subsection~\ref{subsec:lyndon}), and
use $\PP^L_k$ for the set of all non-trivial Lyndon permutations of size at most $k$.
In particular, they noted that the dimension of the region of feasible pattern densities
is at most the number of non-trivial Lyndon permutations~\cite{BorP20}, and
conjectured~\cite[Conjecture 1.3]{BorP20} that this bound is tight.
We combine methods from algebra and combinatorics to prove this conjecture; note that
the set $\PP^L_k$ of non-trivial Lyndon permutations
is a superset of $\PP^I_k$ of non-trivial indecomposable permutations.

\begin{theorem}
\label{thm:lower}
For every integer $k\ge 2$,
there exists $x_0\in [0,1]^{\PP^L_k}$ and $\varepsilon>0$ such that
for every $x\in B_\varepsilon(x_0)\subseteq [0,1]^{\PP^L_k}$
there exists a permuton $\Pi$ such that
\[d(\sigma,\Pi)_{\sigma\in\PP^L_k}=x.\]
\end{theorem}

Theorem~\ref{thm:lower} determines
the dimension of the feasible region of densities of patterns of size at most $k$ in large permutations.
Since the statement that
there exists a polynomial function $f:[0,1]^{\PP^L_k}\to [0,1]^{\PP_k}$ such that,
for every permuton $\Pi$,
we have
$f(d(\sigma,\Pi)_{\sigma\in\PP^L_k})=d(\sigma,W)_{\sigma\in\PP_k}$,
is given in~\cite{BorP20} without particular details,
we give a direct proof of the existence of the function $f$ in Section~\ref{sec:upper} (Theorem~\ref{thm:upper}) for completeness.
The presented proof (unlike the one suggested in~\cite{BorP20}) uses only basic properties of Lyndon words and flag algebras,
which are surveyed in Subsections~\ref{subsec:lyndon} and~\ref{subsec:flag}.
We remark that
properties of Lyndon words seem to capture independence of ``order-like'' combinatorial structures and
the ideas presented in this paper were followed in~\cite{KraLPS23}
to determine the dimension of the feasible region of densities of subtournaments in large tournaments.

\section{Preliminaries}
\label{sec:prelim}

We now introduce notation used throughout this manuscript.
As we combine techniques from several different areas,
we introduce the relevant techniques from each of the areas in one of the subsections that follow.
Before doing so, we first fix some general notation.
The set of the first $n$ positive integers is denoted by $[n]$, and
if $A$ is a set of integers and $k\in\NN$, then $A+k$ is the set $\{a+k,a\in A\}$.
A subset $A$ of integers is an \emph{interval}
if $A$ consists of consecutive integers,
i.e., there exist integers $a,b\in\NN$, $a\le b$, such that $A=\{a,a+1,\ldots,b\}$.
Finally, two subsets $I$ and $J$ of integers
are \emph{non-crossing} if $\max(I)<\min(J)$ or vice versa.

\subsection{Graph limits}
\label{subsec:limg}

We now introduce basic notations related to graph limits as developed
in particular in~\cite{BorCLSSV06,BorCLSV08,BorCLSV12,LovS06,LovS10};
we refer to the monograph by Lov\'asz~\cite{Lov12} for a comprehensive introduction to graph limits.
Given two graphs $H$ and $G$, we say a map $\phi:V(H)\rightarrow V(G)$ is a \emph{homomorphism}
if it preserves edges, i.e. if $\phi(x)\phi(y)\in E(G)$ whenever $xy\in E(H)$.
The \emph{homomorphic density} of $H$ in $G$, denoted by $t(H,G)$,
is the probability that a uniform random function $f:V(H)\to V(G)$,
is a homomorphism of $H$ to $G$.
A sequence $(G_n)_{n\in\NN}$ of graphs is \emph{convergent}
if the number of vertices of $G_n$ tends to infinity and, for every fixed $H$,
the sequence of densities $t(H,G_n)$ converges as $n$ tends to infinity.

Convergent sequences of graphs are represented by an analytic object called a \emph{graphon}:
this is a symmetric measurable function $W:[0,1]^2\to [0,1]$,
i.e., $W(x,y)=W(y,x)$ for $(x,y)\in [0,1]^2$.
The \emph{homomorphic density} of a graph $H$ in a graphon $W$ is defined by
\[t(H,W)=\int_{[0,1]^{V(H)}}\prod_{uv\in E(H)}W(x_u,x_v)\dd x_{V(H)}.\]
We often just briefly say the \emph{density} of a graph $H$ in $W$ rather than
the homomorphic density of $H$ in $W$.
A graphon $W$ is a \emph{limit} of a convergent sequence $(G_n)_{n\in\NN}$ of graphs
if, for every graph $H$, $t(H,W)$ is the limit of $(t(H,G_n))_{n\in\mathbb N}$.
Every convergent sequence of graphs has a limit graphon and
every graphon is a limit of a convergent sequence of graphs as shown by Lov\'asz and Szegedy~\cite{LovS06};
also see~\cite{DiaJ08} for a relation to exchangeable arrays.

\subsection{Permutations}

A \emph{permutation} of \emph{size} $n$ is a bijective function $\pi$ from $[n]$ to $[n]$;
the size of $\pi$ is denoted by $\lvert\pi\rvert$.
The permutation $\pi$ is represented as a word $\pi(1)\pi(2)\dots\pi(n)$,
e.g., $123$ stands for the identity permutation of size $3$.
We say that a permutation is \emph{non-trivial} if its size is at least two.
The set of all permutations of size at most $k$ is denoted by $\PP_k$,
e.g., $\PP_3=\{1,12,21,123,132,213,231,312,321\}$.

The \emph{direct sum} of two permutations $\pi_1$ and $\pi_2$ is the permutation $\pi$ of size $\lvert\pi_1\rvert+\lvert\pi_2\rvert$ such that
$\pi(k)=\pi_1(k)$ for $k\in[\lvert\pi_1\rvert]$ and $\pi(\lvert\pi_1\rvert+k)=\lvert\pi_1\rvert+\pi_2(k)$ for $k\in[\lvert\pi_2\rvert]$;
the permutation $\pi$ is denoted by $\pi_1\oplus\pi_2$.
A permutation is \emph{indecomposable} if it is not a direct sum of two permutations.
Every permutation $\pi$ is a direct sum of indecomposable permutations (possibly with repetitions),
which are referred to as the \emph{indecomposable blocks} of the permutation $\pi$.
For example, the permutation $321645987=321\oplus 312\oplus 321$ consists of three indecomposable blocks.
An \emph{increasing segment} of a permutation $\pi$
is an inclusion-wise maximal interval $A\subseteq [\lvert\pi\rvert]$ such that
$\pi(a+1)=\pi(a)+1$ for every $a\in A$ such that $a+1\in A$.
For example, the permutation $312456$ consists of three increasing segments.

The \emph{pattern} induced by elements $1\le k_1<\cdots<k_m\le n$ in a permutation $\pi$
is the unique permutation $\sigma:[m]\to [m]$ such that $\sigma(i)<\sigma(i')$
if and only if $\pi(k_i)<\pi(k_{i'})$ for all $i,i'\in [m]$.
The \emph{density} of a permutation $\sigma$ in a permutation $\pi$, denoted by $d(\sigma,\pi)$,
is the probability that the pattern induced by $\lvert\sigma\rvert$ uniformly randomly chosen elements of $[\lvert\pi\rvert]$ in $\pi$
is the permutation $\sigma$.
Similarly to the graph case,
we say that a sequence $(\pi_n)_{n\in\NN}$ of permutations is \emph{convergent}
if the size of $\pi_n$ tends to infinity and, for every fixed $\sigma$, 
the sequence of densities $d(\sigma,\pi_n)$ converges as $n$ tends to infinity.

\subsection{Permutation limits}
\label{subsec:limp}

We now introduce analytic representations of convergent sequences of permutations
as originated in~\cite{HopKMRS13,HopKMS11,KraP13} and
further developed in~\cite{BalHLPUV15,ChaKNPSV20,GleGKK15,KenKRW20,Kur22,GarHKP22}.
A \emph{permuton} is a probability measure $\Pi$ on the $\sigma$-algebra of Borel subsets from $[0,1]^2$ that
has uniform marginals, i.e.,
\[\Pi([a,b]\times[0,1]) = \Pi([0,1]\times[a,b]) = b - a\]
for all $0\le a\le b\le 1$ (equivalently, its projection on each of the two dimensions is the uniform measure).
A \emph{$\Pi$-random permutation} of size $n$ is the permutation $\sigma$ obtained
by sampling $n$ points according to the measure $\Pi$ (note that the $x$-coordinates and $y$-coordinates of the sampled points
are pairwise distinct with probability~$1$),
sorting them according to their $x$-coordinates,
say $(x_1,y_1),\ldots,(x_n,y_n)$ for $x_1<\cdots<x_n$, and
defining $\sigma$ so that $\sigma(i)<\sigma(j)$ if and only if $y_i<y_j$ for $i,j\in [n]$.
The \emph{density} of a permutation $\sigma$ in a permuton $\Pi$,
which is denoted by $d(\sigma,\Pi)$,
is the probability that the $\Pi$-random permutation of size $|\sigma|$ is $\sigma$.
Finally, if $S$ is a formal linear combination of permutations,
then $d(S,\Pi)$ is the linear combination of the densities of its constituents in $\Pi$, with coefficients given by the combination.
For example, if $S=\frac{1}{2}\,12+\frac{1}{3}\,123$, then $d(S,\Pi)=\frac{1}{2}\,d(12,\Pi)+\frac{1}{3}\,d(123,\Pi)$.

A permuton $\Pi$ is a \emph{limit} of a convergent sequence $(\pi_n)_{n\in\NN}$ of permutations
if, for every permutation $\sigma$, $d(\sigma,\Pi)$ is the limit of $d(\sigma,\pi_n)$.
Every permuton is a limit of a convergent sequence of permutations and
every convergent sequence of permutations has a (unique) limit permuton~\cite{HopKMRS13,HopKMS11}.

\begin{figure}
\begin{center}
\epsfbox{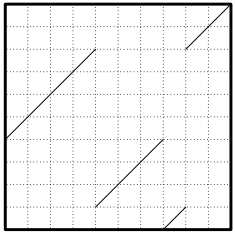}
\end{center}
\caption{The support of the blow-up of the permutation $42315$ scaled by $0.4$, $0.2$, $0.1$, $0.1$ and $0.2$.}
\label{fig:blowup}
\end{figure}

We next define a way of creating a permuton from a given permutation $\pi$.
Let $\pi$ be a permutation of size $k$ and $z_1,\ldots,z_k$ non-negative reals summing to one.
The \emph{blow-up} of the permutation $\pi$ with parts \emph{scaled} by the factors $z_1,\ldots,z_k$
is the unique permuton $\Pi$ defined as follows (an example is given in Figure~\ref{fig:blowup}).
Let $s_i$, $i\in [k]$, be the sum of $z_1+\cdots+z_{i-1}$ and
let $t_i$, $i\in [k]$, be the sum of $z_{\pi^{-1}(1)}+\cdots+z_{\pi^{-1}(i-1)}$.
The support of the permuton $\Pi$ is the set
\[\bigcup_{i\in [k]}\{(s_i+x,t_{\pi(i)}+x),0\le x\le z_i\}.\]
Informally speaking, the support of the permuton $\Pi$ is formed by $k$ increasing segments that
are arranged in the order given by the permutation $\pi$ and scaled by the factors $z_1,\ldots,z_k$.
Since the set
\[\{x\subseteq [0,1] \  |\  \exists \ y\neq y' \text{ with} \ (x, y), \ (x, y')\in \supp(\Pi)\}\]
has measure zero,
the support uniquely determines a probability measure on $[0,1]^2$ with uniform marginals;
this probability measure is the permuton $\Pi$.
In particular, $\Pi(X)$ for a Borel subset $X\subseteq [0,1]^2$
is equal to the measure of the projection of the set $X\cap\supp(\Pi)$ on either of the two coordinates (the measure
of either of the projections is the same).

\subsection{Lyndon words and Lyndon permutations}
\label{subsec:lyndon}

Lyndon words, introduced by \v{S}ir\v{s}ov~\cite{Shi53} and by Lyndon~\cite{Lyn54} in the 1950s,
form a notion that has nowadays many applications in algebra, combinatorics and computer science.
Let $\Sigma$ be a linearly ordered alphabet.
The lexicographic order on words over $\Sigma$ is denoted by $\preceq$.
A word $s_1\dots s_n$ over the alphabet $\Sigma$ is a \emph{Lyndon word}
if no proper suffix of the word $s_1\cdots s_n$
is smaller (in the lexicographic order) than the word $s_1\dots s_n$ itself.
For example, the word $aab$ is Lyndon but the word $aba$ is not (with respect to the usual order on letters).
The following well-known property of Lyndon words~\cite{CheFL58,Rad79,Pen22} is important for our arguments.
\begin{proposition}
\label{prop:lyndon}
Every word over a linearly ordered alphabet
can be uniquely expressed as a concatenation of Lyndon words $w_1,\ldots,w_{\ell}$ such that
$w_1\succeq\cdots\succeq w_{\ell}$.
\end{proposition}
For example, the word $aba$ is the concatenation of the Lyndon words $ab$ and $a$, while
the word $ababa$ is the concatenation of the Lyndon words $ab$, $ab$ and $a$.

A \emph{subword} of a word $s_1\dots s_n$ is any word $s_{i_1}\dots s_{i_m}$ with $1\le i_1<\cdots<i_m\le n$.
The \emph{shuffle product} of words $s_1\dots s_n$ and $t_1\dots t_m$,
which is denoted by $s_1\dots s_n\otimes_S t_1\dots t_m$,
is the formal sum of all $\binom{n+m}{n}$ (not necessarily distinct) words of length $n+m$ that
contain the words $s_1\dots s_n$ and $t_1\dots t_m$ as subwords formed by disjoint sets of letters.
For example, $ab\otimes_S ac=2\,aabc+2\,aacb+abac+acab$.
The following statement can be found in~\cite[Theorem 3.1.1(a)]{Rad79}.

\begin{lemma}
\label{lm:shuffle}
Let $w_1,\ldots,w_n$ be Lyndon words such that $w_1\succeq\cdots\succeq w_n$.
The lexicographically largest constituent in the shuffle product $w_1\otimes_S\cdots\otimes_S w_n$ is the term $w_1\dots w_n$, and
if the words $w_1,\ldots,w_n$ are pairwise distinct, then the coefficient of this term is equal to one.
\end{lemma}

In this manuscript,
we always work with the alphabet $\Sigma$ of indecomposable permutations,
i.e., the letters of the alphabet are indecomposable permutations and the linear order is defined in a way that
indecomposable permutations of smaller size precede those of larger size and
indecomposable permutations of the same size are ordered lexicographically.
Hence, the first five letters of $\Sigma$ are the following five (indecomposable) permutations:
$1$, $21$, $231$, $312$ and $321$ (in this order).
Given a permutation $\pi$,
we write $\overline{\pi}$ for the word over the alphabet $\Sigma$ consisting of the indecomposable blocks of $\pi$.
We write $\pi<_L\pi'$ if the word $\overline{\pi}$ is lexicographically smaller than the word $\overline{\pi'}$, and
use the symbols $\le_L$, $>_L$ and $\ge_L$ in regard to this order in the usual sense.
In particular, $1<_L 12<_L 132<_L 21<_L 231$.
A permutation $\pi$ is a \emph{Lyndon permutation} if the word $\overline{\pi}$ is a Lyndon word.
For example, the permutation $21\oplus 231=21453$ is a Lyndon permutation
but the permutations $12=1\oplus 1$, $213=21\oplus 1$ and $2143=21\oplus 21$ are not.
Note that all indecomposable permutations are Lyndon.
The set of all non-trivial Lyndon permutations of size at most $k$ is denoted by $\PP^L_k$,
e.g., $\PP^L_2=\{21\}$ and $\PP^L_3=\{21,132,231,312,321\}$.

The following is a folklore result; we include its short proof for completeness.

\begin{proposition}
The number $\lvert\PP_k^L\rvert$ of non-trivial Lyndon permutations of length at most $k$
is independent of the choice of ordering of the alphabet $\Sigma$ of indecomposable permutations.
\end{proposition}
\begin{proof}
Fix any ordering of $\Sigma$ and
let $\ell_k$ be the number of Lyndon permutations of size exactly $k$ that arise from this ordering.
Proposition~\ref{prop:lyndon} implies the following power series identity:
\[ \prod_{n\geq 1} (1-x^n)^{-\ell_n} = \sum_{n\geq 1} n! x^n.\]
This uniquely determines each $\ell_n$, and hence $|\PP_k^L|$.
\end{proof}

\subsection{Flag algebra products}
\label{subsec:flag}

The flag algebra method of Razborov~\cite{Raz07} catalyzed progress on many important problems in extremal combinatorics and
has been successfully applied to problems concerning
graphs~\cite{Raz08,Grz12,HatHKNR12,KraLSWY13,HatHKNR13,PikV13,BabT14,BalHLL14,PikR17,GrzHV19,PikST19},
digraphs~\cite{HlaKN17,CorR17}, hypergraphs~\cite{Raz10,BabT11,GleKV16,BalCL22},
geometric settings~\cite{KraMS12}, permutations~\cite{BalHLPUV15,SliS18,CruDN23}, and
other combinatorial objects.
We will not introduce the method completely but we present the concept of a product,
which will be important in our further considerations.
To avoid confusion with other types of products considered in this manuscript,
we refer to the product as a flag product.
If $\pi_1$ and $\pi_2$ are two permutations of sizes $k_1$ and $k_2$, respectively,
then the \emph{flag product} of $\pi_1$ and $\pi_2$, which is denoted by $\pi_1\times\pi_2$,
is a formal linear combination of all permutations $\sigma$ of size $k_1+k_2$ such that
there exists a $k_1$-element set $S\subseteq [k_1+k_2]$ such that
the pattern induced by $S$ in $\sigma$ is $\pi_1$ and
the pattern induced by $[k_1+k_2]\setminus S$ in $\sigma$ is $\pi_2$;
the coefficient at $\sigma$ is equal to the number of choices of such $S$ divided by $\binom{k_1+k_2}{k_1}$.
We give an example:
\begin{equation}
12\times 1=\frac{3}{3}\,123+\frac{2}{3}\,132+\frac{1}{3}\, 231+\frac{2}{3}\, 213+\frac{1}{3}\, 312.
\label{eq:ex1}
\end{equation}
The following identity follows from \cite{Raz07} and holds for any permuton $\Pi$ and any permutations $\pi_1,\ldots,\pi_n$:
\begin{equation}
d\left(\pi_1\times\cdots\times\pi_n,\Pi\right)=d(\pi_1,\Pi)d(\pi_2,\Pi)\cdots d(\pi_n,\Pi).
\label{eq:product}
\end{equation}
For example, \eqref{eq:ex1} and \eqref{eq:product} yield that
the following identity holds for any permuton $\Pi$:
\[
d(1,\Pi)d(12,\Pi)=\frac{3}{3}d(123,\Pi)+\frac{2}{3}d(132,\Pi)+\frac{1}{3}d(231,\Pi)+\frac{2}{3}d(213,\Pi)+\frac{1}{3}d(312,\Pi).
\]

\section{Lower bound}
\label{sec:lower}

In this section, we prove our lower bound on the dimension of the feasible region of densities of $k$-patterns;
the matching upper bound is proven in the next section.
We start with the following lemma, which is a key part of the proof of Theorem~\ref{thm:lower}.
The main idea of the proof of the lemma is similar to that used in the proof of Lemma~\ref{lm:flag}.

\begin{lemma}
\label{lm:lyndon}
Let $\pi_1,\ldots,\pi_n$ be Lyndon permutations such that $\pi_1>_L\pi_2>_L\cdots>_L\pi_n$, and
let $N$ be the size of the permutation $\pi_1\oplus\cdots\oplus\pi_n$.
If $J_1,\ldots,J_n$ are disjoint subsets of $[N]$ such that
the pattern induced by $J_i$ in $\pi_1\oplus\cdots\oplus\pi_n$  is $\pi_i$ for every $i\in [n]$,
then every $J_i$ is an interval and $J_i=[\lvert\pi_i\rvert]+\lvert\pi_1\rvert+\cdots+\lvert\pi_{i-1}\rvert$ for every $i\in [n]$.
\end{lemma}

\begin{proof}
Fix the permutations $\pi_1,\ldots,\pi_n$ and the subsets $J_1,\ldots,J_n$ as in the statement of the lemma.
For $i\in [n]$,
let $m_i$ be the number of indecomposable blocks of $\pi_i$ and
let $\sigma_{i,1},\ldots,\sigma_{i,m_i}$ be indecomposable permutations such that
$\pi_i=\sigma_{i,1}\oplus\cdots\oplus\sigma_{i,m_i}$.
Note that the permutation $\pi_1\oplus\cdots\oplus\pi_n$ has $m_1+\cdots+m_n$ indecomposable blocks.
Further, let $J_{i,1},\ldots,J_{i,m_i}$ be the unique partition of $J_i$ into non-crossing disjoint sets such that
$J_{i,j}$ contains the indices that induce the pattern $\sigma_{i,j}$, $j\in [m_i]$.
Since the pattern induced by $J_{i,j}$, $i\in [n]$ and $j\in [m_i]$, is an indecomposable permutation,
the indices contained in $J_{i,j}$ are contained in the same indecomposable block of the permutation $\pi_1\oplus\cdots\oplus\pi_n$.
Since the sets $J_{i,j}$, $i\in [n]$ and $j\in [m_i]$, partition the set $[N]$ and
the number of indecomposable blocks of $\pi_1\oplus\cdots\oplus\pi_n$ is $m_1+\cdots+m_n$,
it follows that each set $J_{i,j}$, $i\in [n]$ and $j\in [m_i]$,
is the set of indices of one of the indecomposable blocks of $\pi_1\oplus\cdots\oplus\pi_n$.
In particular, indecomposable blocks of the permutations $\pi_1,\ldots,\pi_n$
one-to-one correspond to indecomposable blocks of the permutation $\pi_1\oplus\cdots\oplus\pi_n$.
Since $\overline{\pi_1\oplus\cdots\oplus\pi_n}$
is the lexicographically largest constituent in the shuffle product $\overline{\pi_1}\otimes_S\cdots\otimes_S\overline{\pi_n}$ and
the coefficient at this constituent is one by Lemma~\ref{lm:shuffle} (note that $\pi_1,\ldots,\pi_n$ are pairwise distinct),
it follows that
each $J_i$, $i\in [n]$, is the set of indices corresponding to $\pi_i$ in $\pi_1\oplus\cdots\oplus\pi_n$.
The statement of the lemma now follows.
\end{proof}

We are now ready to prove our main result, Theorem~\ref{thm:lower}.

\begin{proof}[Proof of Theorem~\ref{thm:lower}]
Fix an integer $k\ge 2$.
Let $N$ be the number of non-trivial Lyndon permutations of size at most $k$, i.e. $N=\lvert\PP^L_k\rvert$, and
let $\pi_1,\ldots,\pi_{N+1}$ be all Lyndon permutations of size at most $k$ listed in a way that $\pi_1>_L\cdots>_L\pi_{N+1}$;
note that $\pi_1=k(k-1)\dots 1$ and $\pi_{N+1}=1$.
Let $n_i$ be the size of the permutation $\pi_i$, $i\in [N+1]$.

We next define a family of permutons parameterized by $N+(n_{1}+\cdots+n_{N})$ variables,
namely by $s_1,\ldots,s_{N}$ and $t_{i,1},\ldots,t_{i,n_{i}}$ for $i\in [N]$. For brevity, we will sometimes refer to $s_1, \dots, s_N$ as the \textit{$s$-variables} and to $t_{1, 1}, \dots, t_{N, n_N}$ as the \textit{$t$-variables}.
The parameters will be positive reals such that $s_1+\cdots+s_N<1$ and
$t_{i,1}+\cdots+t_{i,n_{i}}<1$ for every $i\in [N]$.
The permuton $\Pi^L(s_1,\ldots,s_{N},t_{1,1},\ldots,t_{N,n_{N}})$
is the blow-up of the permutation $\pi_1\oplus\cdots\oplus\pi_{N+1}$
such that the $n_1+\cdots+n_{N+1}$ parts of the blow-up are scaled
by the factors
\[s_1t_{1,1},\ldots,s_1t_{1,n_1},\; s_2t_{2,1},\ldots,s_2t_{2,n_2},\; \ldots,\; s_Nt_{N,1},\ldots,s_Nt_{N,n_N},\; z\]
where $z= 1-\sum\limits_{i=1}^{N}\sum\limits_{j=1}^{n_i}s_it_{i, j}$.
We illustrate the construction of a permuton $\Pi^L$ for $k=3$.
There are five non-trivial Lyndon permutations of size at most $3$ (and so six Lyndon permutations of size at most $3$ in total),
i.e., $N=5$, $\pi_1=321$, $\pi_2=312$, $\pi_3=231$, $\pi_4=21$, $\pi_5=132$ and $\pi_6=1$.
Hence, the permuton $\Pi^L$ is parameterized
by $19$ variables $s_1,\ldots,s_5$ and $t_{i,j}$ for $(i,j)\in [5]\times [3]\setminus \{(4,3)\}$.
The permuton $\Pi^L$ itself is visualized in Figure~\ref{fig:PiL3}.

\begin{figure}
\begin{center}
\epsfbox{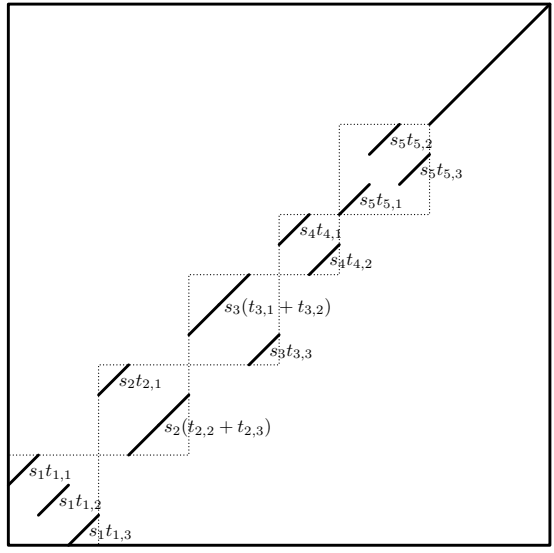}
\end{center}
\caption{The permuton $\Pi^L_3$.}
\label{fig:PiL3}
\end{figure}

Fix $i\in [N]$ and consider the permutation $\pi_i$.
Observe that the probability that
$n_i$ points sampled based on the permuton $\Pi^L$ form the pattern $\pi_i$,
conditioned on at least one point sampled from the part corresponding to $\pi_{N+1}$,
is zero, as none of the permutations $\pi_1,\ldots,\pi_N$ ends with an indecomposable block of size one.
It follows that the density $d(\pi_i,\Pi^L)$ is a homogeneous polynomial of order $2n_i$ such that
each term of the polynomial contains $n_i$ variables $s_1,\ldots,s_N$ (possibly with multiplicities) and
$n_i$ variables $t_{1,1},\ldots,t_{N,n_N}$ (again possibly with multiplicities);
each monomial corresponds to one of the possible choices of points from parts of the permuton $\Pi^L$ that yields $\pi_i$.
Moreover, if each of the $t$-variables appears once in the monomial,
then the pattern of $\pi_1\oplus\cdots\oplus\pi_N$ induced by the elements corresponding to the $t$-variable of the monomial
is $\pi_i$, and
every pattern of $\pi_1\oplus\cdots\oplus\pi_N$ that is $\pi_i$ is associated with one such monomial.

We now analyze the Jacobian matrix $\mathcal{J}$ of the function $\left(d(\pi_1,\Pi^L),\ldots,d(\pi_N,\Pi^L)\right)$
with respect to the variables $s_1,\ldots,s_N$,
i.e.
\[\mathcal{J}_{i, j}=\frac{\partial}{\partial s_j}d\left(\pi_i,\Pi^L(s_1,\ldots,s_{N},t_{1,1},\ldots,t_{N,n_{N}})\right).\]
Since the density $d(\pi_i,\Pi^L)$ is a homogeneous polynomial of order $2n_i$ with $n_i$ $s$-variables and $n_i$ $t$-variables,
the Jacobian determinant
\begin{equation}
\det(\mathcal{J})=\sum_{\sigma\in S_N} \sgn(\sigma) \prod_{i=1}^N \mathcal{J}_{i, \sigma(i)}
\label{eq:detJ}
\end{equation}
is a homogeneous polynomial of order $2(n_1+\cdots+n_N)-N$ such that
each of its summands is a monomial with $n_1+\cdots+n_N-N$ $s$-variables and
$n_1+\cdots+n_N$ $t$-variables (counted with multiplicity, in both cases).
We next investigate those monomials of the Jacobian determinant that
contain all of the $n_1+\cdots+n_N$ $t$-variables,
i.e. each of the $t$-variables appears once in the monomial.

Consider a permutation $\sigma$ in \eqref{eq:detJ} and
note that the monomials that occur in $\prod_{i=1}^N\mathcal{J}_{i, \sigma(i)}$ arise as
products of monomials that occur in each of the $\mathcal{J}_{i, \sigma(i)}$.
For such a product of monomials,
let $I_i$ be the multiset of (double) indices of the variables $t_{1,1},\ldots,t_{N,n_N}$
contained in the monomial occurring in $\mathcal{J}_{i, \sigma(i)}$.
Recall that each monomial of the polynomial $d(\pi_i,\Pi^L)$
corresponds to one of the possible choices of points from parts of the permuton $\Pi^L$ that yields $\pi_i$, and
the chosen parts are determined by the variables contained in the monomial, and
we never use the part of $\Pi^L$ corresponding to $\pi_{N+1}$ (as argued earlier).
If the term in the product in \eqref{eq:detJ} associated with $I_1,\ldots,I_N$
contains each of the $t$-variables exactly once,
then the term corresponds to sampling at most one point from each part of $\Pi^L$ and
so no double index occurs twice in $I_1\cup\cdots\cup I_N$ (and each $I_i$ is actually a set rather than just a multiset).
Since an index $(a, b)\in I_i$ corresponds to the $b$-th element of the permutation $\pi_a$
in $\pi_1\oplus\cdots\oplus\pi_{N}$,
we can map each set $I_i$ to a set $J_i\subseteq [ n_1+\cdots +n_N]$ of (single) indices
indicating the elements of $\pi_1\oplus\cdots\oplus\pi_{N+1}$ inducing the pattern $\pi_i$;
the mapping is simply $(a, b) \rightarrow \sum_{j=1}^{a-1}|\pi_j| + b$.
Hence, if the term in the product in \eqref{eq:detJ} associated with $I_1,\ldots,I_N$
contains each of the $t$-variables exactly once,
the sets $J_1,\ldots,J_N$ form a partition of $[n_1+\ldots+n_N]$, and
the pattern of $\pi_1\oplus\cdots\oplus\pi_N$ induced by $J_i$ is $\pi_i$ for every $i\in [N]$.

Since the sets $J_1,\ldots,J_N$ are disjoint and
the pattern of $\pi_1\oplus\cdots\oplus\pi_N$ induced by $J_i$ is $\pi_i$ for every $i\in [N]$,
Lemma~\ref{lm:lyndon} yields that
each of the sets $J_i$ is an interval that
corresponds to the entries of $\pi_i$ in the permutation $\pi_1\oplus\cdots\oplus\pi_N$.
It follows that the only permutation $\sigma$ such that $\prod_{i=1}^{N}\mathcal{J}_{i, \sigma(i)}$
yields a monomial term containing all $t$-variables is the identity permutation of size $N$, and
the said monomial is obtained precisely by multiplying the terms $s_i^{n_i-1}t_{i,1}\cdots t_{i,n_i}$
in the polynomials $\frac{\partial}{\partial s_i}d\left(\pi_i,\Pi^L\right)$, $i\in [N]$.
Note that the coefficient of the term $s_i^{n_i}t_{i,1}\cdots t_{i,n_i}$ in $d\left(\pi_i,\Pi^L\right)$
is equal to
\[\ell_{i,1}!\ell_{i,2}!\cdots\ell_{i,m_i}!\]
where $m_i$ is the number of increasing segments of $\pi_i$ and $\ell_{i,1},\ldots,\ell_{i,m_i}$ are their lengths;
in particular, the coefficient is non-zero.
We conclude that the coefficient of the term
\[\prod_{i=1}^N s_i^{n_i-1}\prod_{j=1}^{n_i}t_{i,j}\]
in the determinant of $\mathcal{J}$ is non-zero.

Since the Jacobian determinant is a polynomial that is not identically zero,
there exists a choice of positive reals $s_1,\ldots,s_N$, $s_1+\cdots+s_N<1$, and
positive reals $t_{1,1},\ldots,t_{N,n_N}$, $t_{i,1}+\cdots+t_{i,n_{i}}<1$ for every $i\in [N]$, such that
the Jacobian determinant is non-zero.
It follows using the Inverse Function Theorem that
the point $x_0\in [0,1]^{\PP^L_k}$ such that
\[x_0=d\left(\pi,\Pi^L(s_1,\ldots,s_{N},t_{1,1},\ldots,t_{N,n_{N}})\right)_{\pi\in\PP^L_k}\]
satisfies the statement of the theorem.
\end{proof}

\section{Upper bound}
\label{sec:upper}

We start by proving a lemma on the interplay of the flag product of permutations and Lyndon permutations.
Note that Proposition~\ref{prop:lyndon} ensures that
the decomposition of any permutation $\pi$ into $\pi_1,\ldots,\pi_n$ as described in the statement of the next lemma
always exists and is unique.

\begin{lemma}
\label{lm:flag}
Let $\pi$ be a permutation and
let $(\pi_1,\ldots,\pi_n)$ be the unique ordered tuple of permutations such that 
\begin{itemize}
\item $\pi=\pi_1\oplus\cdots\oplus\pi_n$,
\item the words $\overline{\pi_1},\ldots,\overline{\pi_n}$ are Lyndon, and
\item the sequence $\overline{\pi_1},\ldots,\overline{\pi_n}$ is lexicographically non-increasing, i.e., $\pi_1\ge_L\cdots\ge_L\pi_n$.
\end{itemize}
If a permutation $\sigma$ is a constituent in the flag product $\pi_1\times\cdots\times\pi_n$ and $\sigma\not=\pi$,
then either
\begin{itemize}
\item $\sigma$ has fewer indecomposable blocks than $\pi$, or
\item $\sigma$ and $\pi$ have the same number of indecomposable blocks and
      $\overline{\sigma}$ is lexicographically smaller than $\overline{\pi}$.
\end{itemize}
\end{lemma}

\begin{proof}
Fix permutations $\pi_1,\ldots,\pi_n$.
Let $N$ be the sum of the sizes of $\pi_1,\ldots,\pi_n$, and
let $\sigma$ be a constituent in the flag product $\pi_1\times\cdots\times\pi_n$.
Hence, there exists a partition of $[N]$ into sets $J_1,\ldots,J_n$ such that
the pattern induced by $J_i$ in $\sigma$ is $\pi_i$, $i\in [n]$.
For $i\in [n]$,
let $m_i$ be the number of indecomposable blocks of $\pi_i$ and
let $J_{i,1},\ldots,J_{i,m_i}$ be the unique partition of $J_i$ into pairwise non-crossing disjoint sets such that
$J_{i,j}$, $j\in [m_i]$, contains the indices that induce in $\sigma$ the $j$-th indecomposable block of $\pi_i$.
Observe that each set $J_{i,j}$, $i\in [n]$ and $j\in [m_i]$,
is a subset of an indecomposable block of $\sigma$.
Hence,
either two different sets $J_{i,j}$ belong to the same indecomposable block of $\sigma$ and
so $\sigma$ has fewer indecomposable blocks than $\pi$ and we arrive at the first conclusion of the lemma, or
each set $J_{i,j}$, $i\in [n]$ and $j\in [m_i]$, belongs to a different indecomposable block of $\sigma$.

In the latter case,
since the sets $J_{i,j}$ partition $[N]$,
each of them contains indices of one of the indecomposable blocks of $\sigma$;
in particular, each set $J_{i,j}$ is an interval.
Therefore, the words $\overline{\pi_1\oplus\cdots\oplus\pi_n}$ and $\overline{\sigma}$ have the same length and
consist of the same multiset of letters.
Moreover, since the sets $J_{i,1},\ldots,J_{i,m_i}$ appear in the order of their second indices for every $i\in [n]$,
$\overline{\sigma}$ is a constituent in the shuffle product $\overline{\pi_1}\otimes_S\cdots\otimes_S\overline{\pi_n}$.
By Lemma~\ref{lm:shuffle},
the lexicographically largest constituent in the shuffle product $\overline{\pi_1}\otimes_S\cdots\otimes_S\overline{\pi_n}$
is $\overline{\pi_1\oplus\cdots\oplus\pi_n}=\overline{\pi}$, and
so $\overline{\sigma}$ is lexicographically smaller than $\overline{\pi}$ unless $\sigma=\pi$.
\end{proof}

We are now ready to prove the main theorem of this section.
As discussed in Section~\ref{sec:intro},
the statement was presented by Borga and the last author in~\cite{BorP20} with a sketch of a possible proof;
we now present a different (in our view more elementary) proof for completeness.

\begin{theorem}
\label{thm:upper}
For every integer $k\ge 2$,
there exists a polynomial function $f:[0,1]^{\PP^L_k}\to [0,1]^{\PP_k}$ such that
\[f\left(\left(d(\sigma,\Pi)\right)_{\sigma\in\PP^L_k}\right)=(d(\pi,\Pi))_{\pi\in\PP_k}\]
for every permuton $\Pi$.
\end{theorem}

Theorem~\ref{thm:upper} follows from the next lemma (note that $d(1,\Pi)=1$ for every permuton $\Pi$).

\begin{lemma}
\label{lm:upper}
Let $\pi$ be a permutation of size $k\ge 2$.
There exists a polynomial $p_{\pi}$ in variables $x_{\sigma}$ indexed by non-trivial Lyndon permutations $\sigma$ of size at most $k$ such that,
for every permuton $\Pi$,
the density $d(\pi,\Pi)$ of $\pi$ in $\Pi$
is equal to the value of $p_{\pi}$ evaluated at $\vec x=(d(\sigma,\Pi))_{\sigma\in \PP_k^L}$.
\end{lemma}

\begin{proof}
Consider the following linear order on all permutations:
$\pi<\tau$ if either $\lvert\pi\rvert<\lvert\tau\rvert$,
$\lvert\pi\rvert=\lvert\tau\rvert$ and $\pi$ has fewer indecomposable blocks than $\tau$, or
$\lvert\pi\rvert=\lvert\tau\rvert$, $\pi$ and $\tau$ have the same number of indecomposable blocks and $\pi<_L\tau$.
By slightly abusing notation, we extend the statement of the lemma to $k=1$ (if $k=1$,
the polynomial $p_1$ depends on no variables, i.e., it is a constant), and
prove the extended statement by induction on the linear order $<$.
The base of the induction is the permutation $\pi=1$ and
we set $p_1$ to be constantly equal to $1$.

We now present the induction step.
Consider a permutation $\pi$ of size $k\ge 2$.
If $\pi$ consists of a single indecomposable block,
then $\pi$ is Lyndon and we set $p_{\pi}(\vec x)=x_{\pi}$.
In the rest, we assume that $\pi$ consists of two or more indecomposable blocks and
let $\pi_1,\ldots,\pi_n$ be the Lyndon permutations such that
$\pi=\pi_1\oplus\cdots\oplus\pi_n$ and $\pi_1\ge_L\cdots\ge_L\pi_n$;
such permutations exist and are unique by Proposition~\ref{prop:lyndon}.
By \eqref{eq:product}, it holds that
\[d(\pi_1\times\cdots\times\pi_n, \Pi)=d(\pi_1,\Pi)d(\pi_2,\Pi)\cdots d(\pi_n,\Pi).\]
By Lemma~\ref{lm:flag}, $d(\pi_1\times\cdots\times\pi_n, \Pi)$
is equal to a linear combination (with fixed coefficients) of densities $d(\sigma,\Pi)$ of permutations $\sigma$ of size $k$
such that $\sigma\leq \pi$.
Note that the coefficient of $d(\pi, \Pi)$ in this linear combination is non-zero.

By induction,
for each permutation $\sigma<\pi$,
we can express $d(\sigma, \Pi)$ as a polynomial in densities of non-trivial Lyndon permutations of size at most $k$.
Substituting these polynomials for all $\sigma<\pi$
in the linear combination of densities $d(\sigma,\Pi)$ that is equal to $d(\pi_1\times\cdots\times\pi_n, \Pi)$,
we obtain an identity that contains a term linear in $d(\pi,\Pi)$ and
terms polynomial in densities of non-trivial Lyndon permutations of size at most $k$.
This yields the existence of the sought polynomial $p_{\pi}$.
\end{proof}

\bibliographystyle{bibstyle}
\bibliography{permdim}

\begin{thebibliography}{10}
\providecommand{\url}[1]{\texttt{#1}}
\providecommand{\urlprefix}{URL }
\providecommand{\eprint}[2][]{\url{#2}}

\bibitem{BabT11}
R.~Baber and J.~Talbot: \emph{Hypergraphs do jump}, Combin. Probab. Comput.
  \textbf{20} (2011), 161--171.

\bibitem{BabT14}
R.~Baber and J.~Talbot: \emph{A solution to the 2/3 conjecture}, SIAM J.
  Discrete Math. \textbf{28} (2014), 756--766.

\bibitem{BalCL22}
J.~Balogh, F.~C. Clemen and B.~Lidick{\`y}: \emph{Solving {T}ur{\'a}n's
  tetrahedron problem for the {$\ell_2$}-norm}, J. Lond. Math. Soc.
  \textbf{106} (2022), 60--84.

\bibitem{BalHLL14}
J.~Balogh, P.~Hu, B.~Lidick{\'y} and H.~Liu: \emph{Upper bounds on the size of
  4-and 6-cycle-free subgraphs of the hypercube}, Eur. J. Combin. \textbf{35}
  (2014), 75--85.

\bibitem{BalHLPUV15}
J.~Balogh, P.~Hu, B.~Lidick\'{y}, O.~Pikhurko, B.~Udvari and J.~Volec:
  \emph{Minimum number of monotone subsequences of length 4 in permutations},
  Combin. Probab. Comput. \textbf{24} (2015), 658--679.

\bibitem{BorP20}
J.~Borga and R.~Penaguiao: \emph{The feasible region for consecutive patterns
  of permutations is a cycle polytope}, Algebr. Comb. \textbf{3} (2020),
  1259--1281.

\bibitem{BorP23}
J.~Borga and R.~Penaguiao: \emph{The feasible regions for consecutive patterns
  of pattern-avoiding permutations}, Discrete Math. \textbf{346} (2023),
  113219.

\bibitem{BorCLSSV06}
C.~Borgs, J.~Chayes, L.~Lov{\'a}sz, V.~T. S{\'o}s, B.~Szegedy and
  K.~Vesztergombi: \emph{Graph limits and parameter testing}, Proc. 38th annual
  ACM Symposium on Theory of computing (STOC) (2006), 261--270.

\bibitem{BorCLSV08}
C.~Borgs, J.~T. Chayes, L.~Lov{\'a}sz, V.~T. S{\'o}s and K.~Vesztergombi:
  \emph{Convergent sequences of dense graphs {I}: {S}ubgraph frequencies,
  metric properties and testing}, Adv. Math. \textbf{219} (2008), 1801--1851.

\bibitem{BorCLSV12}
C.~Borgs, J.~T. Chayes, L.~Lov{\'a}sz, V.~T. S{\'o}s and K.~Vesztergombi:
  \emph{Convergent sequences of dense graphs {II}. {M}ultiway cuts and
  statistical physics}, Ann.~of~Math.~(2)  (2012), 151--219.

\bibitem{ChaKNPSV20}
T.~Chan, D.~{Kr\'al'}, J.~A. Noel, Y.~Pehova, M.~Sharifzadeh and J.~Volec:
  \emph{Characterization of quasirandom permutations by a pattern sum}, Random
  Struct. Algor. \textbf{57} (2020), 920--939.

\bibitem{CheFL58}
K.~T. Chen, R.~H. Fox and R.~C. Lyndon: \emph{Free differential calculus, {IV}.
  {T}he quotient groups of the lower central series}, Ann. Math. \textbf{68}
  (1958), 81--95.

\bibitem{CorR17}
L.~N. Coregliano and A.~A. Razborov: \emph{On the density of transitive
  tournaments}, J. Graph Theory \textbf{85} (2017), 12--21.

\bibitem{CruDN23}
G.~Crudele, P.~Dukes and J.~A. Noel: \emph{Six permutation patterns force
  quasirandomness}, preprint arXiv:2303.04776  (2023).

\bibitem{DiaJ08}
P.~Diaconis and S.~Janson: \emph{Graph limits and exchangeable random graphs},
  Rend. Mat. Appl. \textbf{28} (2008), 33--61.

\bibitem{ErdLS79}
P.~{Erd\H{o}s}, L.~Lov\'{a}sz and J.~Spencer: \emph{Strong independence of
  graphcopy functions}, Graph theory and related topics (1979), 165--172.

\bibitem{GarHKP22}
F.~Garbe, J.~Hladk{\'y}, G.~Kun and K.~Pek{\'a}rkov{\'a}: \emph{On
  pattern-avoiding permutons}, Random Struct. Algor. \textbf{65} (2024),
  46--60.

\bibitem{GleGKK15}
R.~Glebov, A.~Grzesik, T.~Klimo\v{s}ov\'a and D.~Kr{\'{a}}l': \emph{Finitely
  forcible graphons and permutons}, J. Combin. Theory Ser. {B} \textbf{110}
  (2015), 112--135.

\bibitem{GleHKKKL17}
R.~Glebov, C.~Hoppen, T.~Klimo{\v s}ov{\'a}, Y.~Kohayakawa, D.~Kr\'al' and
  H.~Liu: \emph{Densities in large permutations and parameter testing},
  European J. Combin. \textbf{60} (2017), 89--99.

\bibitem{GleKV16}
R.~Glebov, D.~Kr{\'a}l’ and J.~Volec: \emph{A problem of {E}rd{\H{o}}s and
  {S}{\'o}s on 3-graphs}, Israel J. Math. \textbf{211} (2016), 349--366.

\bibitem{Grz12}
A.~Grzesik: \emph{On the maximum number of five-cycles in a triangle-free
  graph}, J. Combin. Theory Ser. B \textbf{102} (2012), 1061--1066.

\bibitem{GrzHV19}
A.~Grzesik, P.~Hu and J.~Volec: \emph{Minimum number of edges that occur in odd
  cycles}, J. Combin. Theory Ser. B \textbf{137} (2019), 65--103.

\bibitem{HatHKNR12}
H.~Hatami, J.~Hladk\'y, D.~Kr\'al', S.~Norine and A.~Razborov:
  \emph{Non-three-colourable common graphs exist}, Combin. Probab. Comput.
  \textbf{21} (2012), 734--742.

\bibitem{HatHKNR13}
H.~Hatami, J.~Hladk{\'y}, D.~{Kr{\'a}l'}, S.~Norine and A.~Razborov: \emph{On
  the number of pentagons in triangle-free graphs}, J. Combin. Theory Ser. A
  \textbf{120} (2013), 722--732.

\bibitem{HlaKN17}
J.~Hladk{\'y}, D.~Kr{\'a}l’ and S.~Norin: \emph{Counting flags in
  triangle-free digraphs}, Combinatorica \textbf{37} (2017), 49--76.

\bibitem{HopKMRS13}
C.~Hoppen, Y.~Kohayakawa, C.~G.~T. de~A.~Moreira, B.~R{\'{a}}th and R.~M.
  Sampaio: \emph{Limits of permutation sequences}, J. Combin. Theory Ser. {B}
  \textbf{103} (2013), 93--113.

\bibitem{HopKMS11}
C.~Hoppen, Y.~Kohayakawa, C.~G.~T. de~A.~Moreira and R.~M. Sampaio:
  \emph{Testing permutation properties through subpermutations}, Theor. Comput.
  Sci. \textbf{412} (2011), 3555--3567.

\bibitem{KenKRW20}
R.~Kenyon, D.~Kr\'{a}l', C.~Radin and P.~Winkler: \emph{Permutations with fixed
  pattern densities}, Random Struct. Algor. \textbf{56} (2020), 220--250.

\bibitem{KraLPS23}
D.~Kr\'al', A.~Lamaison, M.~Prorok and X.~Shu: \emph{The dimension of the
  region of feasible tournament profiles}, preprint arXiv:2310.19482  (2023).

\bibitem{KraLSWY13}
D.~{Kr\'al'}, C.-H. Liu, J.-S. Sereni, P.~Whalen and Z.~B. Yilma: \emph{A new
  bound for the 2/3 conjecture}, Combin. Probab. Comput. \textbf{22} (2013),
  384--393.

\bibitem{KraMS12}
D.~{Kr\'al'}, L.~Mach and J.-S. Sereni: \emph{A new lower bound based on
  {G}romov's method of selecting heavily covered points}, Discrete Comput.
  Geom. \textbf{48} (2012), 487--498.

\bibitem{KraP13}
D.~Kr\'{a}l' and O.~Pikhurko: \emph{Quasirandom permutations are characterized
  by 4-point densities}, Geom. Funct. Anal. \textbf{23} (2013), 570--579.

\bibitem{Kur22}
M.~Kure{\v c}ka: \emph{Lower bound on the size of a quasirandom forcing set of
  permutations}, Combin. Probab. Comput. \textbf{31} (2022), 304--319.

\bibitem{Lov12}
L.~Lov\'asz: Large Networks and Graph Limits, \emph{Colloquium Publications},
  volume~60, 2012.

\bibitem{LovS06}
L.~Lov\'asz and B.~Szegedy: \emph{Limits of dense graph sequences}, J. Combin.
  Theory Ser. B \textbf{96} (2006), 933--957.

\bibitem{LovS10}
L.~Lov{\'a}sz and B.~Szegedy: \emph{Testing properties of graphs and
  functions}, Israel J. Math. \textbf{178} (2010), 113--156.

\bibitem{Lyn54}
R.~C. Lyndon: \emph{On {B}urnside's problem}, Trans. Amer. Math. Soc.
  \textbf{77} (1954), 202--215.

\bibitem{Pen22}
R.~Penaguiao: \emph{Pattern hopf algebras}, Ann. Comb. \textbf{26} (2022),
  405--451.

\bibitem{PikR17}
O.~Pikhurko and A.~Razborov: \emph{Asymptotic structure of graphs with the
  minimum number of triangles}, Combin. Probab. Comput. \textbf{26} (2017),
  138--160.

\bibitem{PikST19}
O.~Pikhurko, J.~Slia{\v{c}}an and K.~Tyros: \emph{Strong forms of stability
  from flag algebra calculations}, J. Combin. Theory Ser. B \textbf{135}
  (2019), 129--178.

\bibitem{PikV13}
O.~Pikhurko and E.~R. Vaughan: \emph{Minimum number of k-cliques in graphs with
  bounded independence number}, Combin. Probab. Comput. \textbf{22} (2013),
  910--934.

\bibitem{Rad79}
D.~A. Radford: \emph{A natural ring basis for the shuffle algebra and an
  application to group schemes}, J. Algebra \textbf{58} (1979), 432--454.

\bibitem{Raz07}
A.~A. Razborov: \emph{Flag algebras}, J. Symbolic Logic \textbf{72} (2007),
  1239--1282.

\bibitem{Raz08}
A.~A. Razborov: \emph{On the minimal density of triangles in graphs}, Combin.
  Probab. Comput. \textbf{17} (2008), 603--618.

\bibitem{Raz10}
A.~A. Razborov: \emph{On 3-hypergraphs with forbidden 4-vertex configurations},
  SIAM J. Discrete Math. \textbf{24} (2010), 946--963.

\bibitem{Shi53}
A.~I. {\v S}ir\v{s}ov: \emph{Subalgebras of free {L}ie algebras}, Mat. Sbornik
  N.S. \textbf{33/75} (1953), 441--452.

\bibitem{SliS18}
J.~Slia{\v c}an and W.~Stromquist: \emph{Improving bounds on packing densities
  of 4-point permutations}, Discrete Math. Theor. Comput. Sci \textbf{19}
  (2018).

\bibitem{Var14}
Y.~Vargas: \emph{Hopf algebra of permutation pattern functions}, DMTCS
  Proceedings vol. AT, 26th International Conference on Formal Power Series and
  Algebraic Combinatorics (FPSAC 2014) (2014), 839--850.

\bibitem{Whi32}
H.~Whitney: \emph{The coloring of graphs}, Ann. Math. \textbf{33} (1932),
  688--718.

\end{thebibliography}
\end{document}